\documentclass[12pt,reqno]{amsart}
\usepackage{amssymb}
\usepackage{mathrsfs}
\usepackage{a4}
\usepackage{hyperref}
\usepackage{xr}

\newcommand{\heute}{10 December 2011}

\theoremstyle{plain}
\newtheorem{theorem}{Theorem}[section]

\newtheorem{lemma}[theorem]{Lemma}

\newtheorem{proposition}[theorem]{Proposition}

\theoremstyle{remark}

\newtheorem{hypothesis}[theorem]{Hypothesis}

\newtheorem*{rk}{Remark}

\newenvironment{namedThm}[1]{\begin{trivlist} \item[\hskip\labelsep
\textbf{#1}]\em}{\end{trivlist}}

\newcommand{\enref}[1]{\textup{(\ref{enum:#1})}}
\newcommand{\dashTwo}[1]{\textup{(\ref{two}${}'$)}}

\newcommand{\ignore}[1]{}

\newcommand{\f}[1][p]{\mathbb{F}_{#1}}

\newcommand{\Gro}[1]{Gr\"ob\-ner}

\newcommand{\GL}{\mathit{GL}}

\newcommand{\pp}[1][G,V]{\mathscr{P}(#1)}

\newcommand{\abs}[1]{\left|#1\right|}
\newcommand{\scr}[1]{\mathscr{#1}}

\newcommand{\xc}{\sigma_x}
\newcommand{\yc}{\sigma_y}
\newcommand{\zc}{\sigma_z}
\newcommand{\tc}{\tau}

\begin{document}

\title[Oliver example]{Weak closure and Oliver's $p$-group conjecture}
\author[D.~J. Green]{David J. Green}
\address{Department of Mathematics \\
Friedrich-Schiller-Universit\"at Jena \\ 07737 Jena \\ Germany}
\email{david.green@uni-jena.de}
\author[J.~Lynd]{Justin Lynd}
\address{Department of Mathematics \\ The Ohio State University \\
Columbus, OH 43210}
\email{jlynd@math.ohio-state.edu}
\date{\heute}

\begin{abstract}
To date almost all verifications of Oliver's $p$-group conjecture have proceeded by verifying a stronger conjecture about weakly closed quadratic subgroups. We construct a group of order~$3^{49}$ which refutes the weakly closed conjecture but satisfies Oliver's conjecture.
\end{abstract}

\maketitle

\section{Introduction}
\noindent
Chermak has shown that every saturated fusion system has a unique centric linking system~\cite{Chermak:Localities}\@. Both Chermak and Oliver in his subsequent proof of the higher limits conjecture~\cite{Oliver:Chermak} invoke the general FF-module theorem~\cite{MeierfStellm:GeneralFF}, which depends on the classification of finite simple groups.

In~\cite{Oliver:MartinoPriddyOdd}, Oliver formulated a conjecture about $p$-groups and gave a classification-free proof that it would imply both the higher limits conjecture and Chermak's theorem at odd primes. Oliver's $p$-group conjecture has been verified in several cases, almost all of which involve verifying the stronger Quadratic Conjecture:

\begin{namedThm}{Quadratic Conjecture (Conjecture 1.4~of \cite{ol3})}
Let $p$~be a prime, $G$ a finite $p$-group, and $V$ a faithful
$\f G$-module. If $V$ is an $F$-module then there is a quadratic element in $\Omega_1(Z(G))$.
\end{namedThm}

\noindent
The Quadratic Conjecture implies Oliver's Conjecture by \cite[Theorem 1.5]{ol3}\@. In \cite[Proposition~4.5]{ol3} it is shown that the Quadratic Conjecture in turn follows from another conjecture. Indeed, up till now all verifications of the Quadratic Conjecture have proceeded by verifying this Weakly Closed Conjecture:

\begin{namedThm}{Weakly Closed Conjecture (Conjecture 4.6~of \cite{ol3})}
Let $p$~be a prime, $G$ a finite $p$-group, and $V$ a faithful
$\f G$-module. If there is an elementary abelian subgroup
$1 \neq E \leq G$ which is both quadratic on~$V$ and weakly closed in
$C_G(E)$ with respect to~$G$,
then there are quadratic elements in $\Omega_1(Z(G))$.
\end{namedThm}

\noindent
In this paper we present a counterexample to the Weakly Closed Conjecture which still satisfies the Quadratic Conjecture.

\begin{theorem}
\label{thm:main}
There is a $3$-group $G$ of order $3^{49}$ with the following properties:
\begin{enumerate}
\item
\label{enum:main-1}
The Weakly Closed Conjecture fails for one faithful $\f[3]G$-module $V_0$.
\item
\label{enum:main-2}
Every faithful $\f[3]G$-module $V$ satisfies the Quadratic Conjecture.
\end{enumerate}
\end{theorem}

\noindent
The group $G$ is constructed in Section~\ref{section:group} and the module $V_0$ in Section~\ref{section:module}\@.
Part \enref{main-1} of Theorem~\ref{thm:main} is proved as Theorem~\ref{thm:V}\@. Part~\enref{main-2} is proved as Proposition~\ref{prop:verify} using Proposition~\ref{prop:NA}\@.

\begin{rk}
It may aid the reader if we recall the exact relationship between the Quadratic Conjecture and the original form of Oliver's $p$-group conjecture. Let $S$ be a finite $p$-group. In \cite[Definition~3.1]{Oliver:MartinoPriddyOdd} Oliver defines a characteristic subgroup $\mathfrak{X}(S) \leq S$, and conjectures that $J(S) \leq \mathfrak{X}(S)$ always holds at odd primes. Here $J(S)$ is the Thompson subgroup generated by the elementary abelian subgroups of greatest order. Section~2 of~\cite{ol3} modifies Oliver's construction and introduces a second characteristic subgroup $\mathcal{Y}(S) \leq S$, with $\mathcal{Y}(S) = S$ for $p=2$ and $\mathcal{Y}(S) \leq \mathfrak{X}(S)$ for odd~$p$\@. Hence the conjecture $J(S) \leq \mathcal{Y}(S)$ is a strengthening of Oliver's conjecture. Theorem 1.5~of \cite{ol3} states that the Quadratic Conjecture is equivalent to the conjecture $J(S) \leq \mathcal{Y}(S)$. More specifically, if $S$ is a counterexample to $J(S) \leq \mathcal{Y}(S)$ then $G = S/\mathcal{Y}(S)$ and $V = \Omega_1 (Z(\mathcal{Y}(S)))$ give a counterexample to the Quadratic Conjecture; and if $(G,V)$ is a counterexample to the Quadratic Conjecture then $S = V \rtimes G$ is a counterexample to $J(S) \leq \mathcal{Y}(S)$.

If $p=3$ then the Quadratic Conjecture is equivalent to Oliver's $p$-group conjecture, as the definitions of $\mathfrak{X}(S)$ and $\mathcal{Y}(S)$ coincide for $p=3$.
\end{rk}

\section{Notation and known results}
\noindent
Let $G$ be a finite $p$-group and $V$ a faithful right $\f G$-module. Our notation for commutators is $[x,y] = x^{-1} y^{-1} x y$ and $[x_1,\ldots,x_{n-1},x_n] = [[x_1,\ldots,x_{n-1}],x_n]$. We often view $V,G$ as subgroups of the semidirect product $V \rtimes G$, giving meaning to expressions such as $[v,g]$ and $C_V(H)$ for $v \in V$, $g \in G$ and $H \leq G$. Note that $[v,g] = v(g-1)$ and $C_V(H) = V^H$.

An element $g \in G$ is called quadratic if its action on~$V$ has minimal polynomial $(X-1)^2$. An elementary abelian subgroup $1 \neq E \leq G$ is called quadratic if $[V,E,E] = 0$, or equivalently if every $1 \neq g \in E$ is quadratic (see Corollary 3.2 of~\cite{ol3})\@.
An elementary abelian subgroup $1 \neq E \leq G$ is called an offender if $j_E(V) \geq 1$, where
\[
j_H(V) = \frac{\abs{H} \cdot \abs{C_V(H)}}{\abs{V}}
\]
for any subgroup $H \leq G$. If there is an offender then $V$ is called an $F$-module. Finally the set $\pp$ of best offenders is defined by
\[
\pp = \{ \text{$E \leq G$} \mid \text{$E$ an offender and $j_F(V) \leq j_E(V)$ for all $F \leq E$} \} \, .
\]
Note that every offender contains a best one.

\begin{namedThm}{Timmesfeld's Replacement Theorem} \emph{(\cite[Theorem~2]{Chermak:QuadraticAction})} \\
Let $G$ be a $p$-group and $V$ a faithful $\f G$-module. Suppose that $E \in \pp$. Then $F = C_E([V,E])$ is a quadratic best offender, and $j_F(V) = j_E(V)$.
\end{namedThm}

\begin{namedThm}{Meierfrankenfeld--Stellmacher Lemma} \emph{(\cite[Lemma 2.6]{MeierfStellm:OtherPGV})} \\
Let $G$ be a $p$-group and $V$ a faithful $\f G$-module. For $H,K \leq G$ one has
\[
j_H(V) j_K(V) \leq j_{H \cap K}(V) j_{\langle H,K \rangle}(V)
\]
with equality if and only if $\langle H, K \rangle = HK$ and $C_V(H \cap K) = C_V(H) + C_V(K)$.
\end{namedThm}

\begin{namedThm}{Normal Abelian Lemma} \emph{(\cite[Theorem 1.5]{newol})} \\
Let $G$ be a $p$-group and $V$ a faithful $\f G$-module with no central quadratics. Suppose that $A \trianglelefteq G$ is an abelian normal subgroup. Then no offender lies in~$A$.
\end{namedThm}

\noindent
Finally we require a slight strengthening of \cite[Lemma~4.1]{oliver}\@.

\begin{namedThm}{Descent Lemma (\cite{oliver})}
Suppose that $p$ is an odd prime, that $G \neq 1$ is a finite $p$-group, and that
$V$ is a faithful $\f G$-module. Suppose that $a,b \in G$ are such that $1 \neq 
c := [a,b]$ lies in $C_G(a,b)$. Suppose further that $b$ is quadratic. Then $\langle b,c \rangle$ is quadratic too.
\end{namedThm}

\begin{proof}
$c$ is quadratic by the statement of \cite[Lemma~4.1]{oliver}\@. So by \cite[Corollary 3.2]{ol3} it suffices to prove that $[V,b,c]=0$. This follows from the assertion $\beta \gamma =0$ in the proof of \cite[Lemma~4.1]{oliver}.
\end{proof}

\section{A technical result on offenders}

\begin{hypothesis}
\label{hypo:NA}
Let $G$ be a finite $p$-group.
Suppose that $A,N \trianglelefteq G$ with $A$ abelian, $A \leq N$, $[N,N] \leq A$ und $[N,A] = 1$. Suppose that $V$ is a faithful $\f G$-module and that there are no quadratic elements in $Z(G)$.
Denote by $\mathcal{E}$ the set of all offenders $E \leq N$ such that $\abs{E \colon E \cap A} = p$.
\end{hypothesis}

\begin{proposition}
\label{prop:NA}
Suppose that Hypothesis~\ref{hypo:NA} is satisfied. Then no $E \in \mathcal{E}$ is weakly closed in $C_G(E)$ with respect to $G$.
\end{proposition}

\noindent
Before we prove the proposition we need to establish three lemmas.

\begin{lemma}
\label{lemma:noRankOne}
Let $G$ be a finite $p$-group and $V$ a faithful $\f G$-module with no central quadratics. Then there is no offender $E \leq G$ with $\abs{E} = p$.
\end{lemma}

\begin{proof}
Recall from \cite[Section~3]{ol3} that we write $g \perp h$ if $[g,h] = 1$ and $[V,g,h] = 0$. This relation is symmetric.
Suppose that $E = \langle x \rangle$ is an offender. Since $x$ acts nontrivially and $j_E(V) \geq 1$ it follows that $C_V(E) = C_V(x)$ has codimension~$1$ in the $\f$-vector space~$V$.
Pick $y \neq 1$ to lie in the intersection of $\Omega_1(Z(G))$ with the normal closure of~$E$. If $x \not \perp y$ then $[v_0,y,x] \neq 0$ for some $v_0 \in V$. But then $v_0$ and $[v_0,y]$ are linearly independent in $V/C_V(x)$, since $v \mapsto [v,y]$ is nilpotent and $[v_0,y,x] = [v_0,x,y] \neq 0$. This contradicts codimension~$1$. We conclude that $x \perp y$.

Now consider $y^{\perp} = \{g \in G \mid g \perp y \}$. By \cite[Lemma 3.1\,(3)]{ol3} this is a normal subgroup of $C_G(y) = G$. Since $x \in y^{\perp}$ it follows by choice of~$y$ that $y \in y^{\perp}$. So $y$ is quadratic by \cite[Lemma 3.1\,(1)]{ol3}, contradicting the assumption that there are no central quadratics.
\end{proof}

\begin{lemma}
\label{lemma:preNA}
Suppose that Hypothesis~\ref{hypo:NA} is satisfied. Let $E \in \mathcal{E}$ and set $F = E \cap A$. Then
\begin{enumerate}
\item
\label{enum:preNA-1}
$F \neq 1$
\item
\label{enum:preNA-2}
$j_E(V) = 1$, $j_F(V) = p^{-1}$ and $C_V(E)=C_V(F)$.
\item
\label{enum:preNA-3}
$E \in \pp$.
\item
\label{enum:preNA-4}
The set $C_E([V,E]) \setminus A$ is non-empty. Each $x \in C_E([V,E]) \setminus A$ satisfies $E = \langle x, F \rangle$ and $[V,x] \leq C_V(E)$.
\end{enumerate}
\end{lemma}

\begin{proof}
By the Normal Abelian Lemma, $A$ contains no offenders.
\par
\enref{preNA-1}: If $F = 1$ then $\abs{E} = p$, which cannot happen by Lemma~\ref{lemma:noRankOne}\@.
\par
\enref{preNA-2}: Follows by definition of $j_E(V)$ since $E$ offends but $1 \neq F \leq A$ does not.
\par
\enref{preNA-3}: Any offender $E_1 < E$ also lies in~$\mathcal{E}$. Apply~\enref{preNA-2}\@. \par
\enref{preNA-4}:
$E_1 = C_E([V,E])$ is a quadratic offender by Timmesfeld's Replacement Theorem. Since $E_1$ offends we have $E_1 \nleq A$. So $E = \langle x,F \rangle$ for any $x \in E_1 \setminus A$. Moreover for $e \in E$ we have
\[
[v,x,e] = v(x-1)(e-1) = v(e-1)(x-1) = [v,e,x] = 0 \, ,
\]
since $[x,e]=1$ and $[V,E,E_1] = 0$ by construction of $E_1$. So $[V,x] \leq C_V(E)$.
\end{proof}

\begin{lemma}
\label{lemma:preNA2}
Suppose that Hypothesis~\ref{hypo:NA} is satisfied and that $E,E_1 \in \mathcal{E}$ are such that $E \leq E_1$, $E \cap A \neq 1$, and $E$ is weakly closed in $C_G(E)$ with respect to $G$. Then there is $E_2 \in \mathcal{E}$ such that $E_1 \leq E_2$ and $E_2$ is weakly closed in $C_G(E_2)$ with respect to $G$.
\end{lemma}

\begin{proof}
We argue by contradiction and assume that $E_1$ is a largest counterexample. Setting $F = E \cap A$ and $F_1 = E_1 \cap A$ we have $F \neq 1$, $F \leq F_1$ and $E_1 = EF_1$. As $E_1$ is not weakly closed there is $g \in G$ with $E_1^g \neq E_1$ and $[E_1,E_1^g] = 1$. Since $E$ is weakly closed and $E \leq E_1$ it follows that $E^g = E$. Hence $F = F^g \leq F_1^g$ and $F_1^g \nleq E_1$.  Set $E_3 = E_1 F_1^g > E_1$. Applying the Meierfrankenfeld--Stellmacher Lemma to $E_1,F_1^g$ we have $j_{E_3}(V) j_{E_1 \cap F_1^g}(V) \geq 1 \cdot p^{-1}$ by Lemma~\ref{lemma:preNA}\@. So since $1 \neq F \leq E_1 \cap F_1^g \leq A$ the Normal Abelian Lemma says that $j_{E_1 \cap F_1^g}(V) \leq p^{-1}$ and hence $E_3 \in \mathcal{E}$. This contradicts the maximality of~$E_1$.
\end{proof}

\begin{proof}[Proof of Proposition~\ref{prop:NA}]
We argue by contradiction and suppose that $E \in \mathcal{E}$ is weakly closed and of the largest possible order.
By Lemma~\ref{lemma:preNA}\,\enref{preNA-4} we have $E = \langle x, F \rangle$ with $x \in N \setminus A$ quadratic.
Now $E \ntrianglelefteq G$ by the Normal Abelian Lemma, so since $E$ is weakly closed there is a $g \in G$ with $[E,E^g] \neq 1$. Since $E, E^g \leq N$ and $[N,A] = 1$ it follows that $1 \neq [x,x^g] \in [N,N] \leq A$.
Once more using $[N,N] \leq A$ and $[N,A] = 1$ we deduce that $[E,E^{x^g}] = 1$. So since $E$ is weakly closed it follows that $E^{x^g} = E$ and therefore $1 \neq [x,x^g] \in F = A \cap E$. Similarly with $E_1 = E^g$ we have $E_1^x = E_1$, and hence $[x,x^g] \in F^g = A \cap E_1$. So $F \cap F^g \neq 1$.

We now apply the Meierfrankenfeld--Stellmacher Lemma to $E$ and $F^g$. Since  $1 \neq [x,x^g] \in E \cap F^g \leq A$ we have $j_{E \cap F^g}(V) \leq p^{-1}$. Hence $j_{EF^g}(V) \geq 1$ and $EF^g \in \mathcal{E}$\@. Since $E$ is maximal weakly closed by assumption, Lemma~\ref{lemma:preNA2} gives us $EF^g = E$ and hence $F^g = F$. So by Lemma~\ref{lemma:preNA}\,\enref{preNA-2} we have
\[
C_V(E^g) = C_V(F^g) = C_V(F) = C_V(E) \, ,
\]
whence by \enref{preNA-4} we have $[V,x,x^g] \leq [C_V(E),x^g] = [C_V(E^g),x^g] = 0$. And since $[v,x^g] = [vg^{-1},x]g$ we also have $[v,x^g,e^g] = [vg^{-1},x,e]g = 0$, meaning that $[V,x^g] \leq C_V(E^g)$ and hence $[V,x^g,x] \leq [C_V(E^g),x] = [C_V(E),x] = 0$.
So $[V,[x,x^g]] = 0$, contradicting the fact that $V$ is faithful and $[x,x^g] \neq 1$.
\end{proof}

\section{The group}
\label{section:group}
\noindent
The construction we describe in this section was inspired by an example of J.~L. Alperin, see \cite[p.~349]{Huppert:I} or \cite[pp.~324--5]{Glauberman:LargeAbelian}\@.
\medskip

\noindent
Let $p$ be an odd prime and $\Omega$ a finite set with at least two elements. Write $F(\Omega)$ for the free group on~$\Omega$, and let $N(\Omega) \trianglelefteq F(\Omega)$ be the normal subgroup generated by all twofold commutators $[g_1,g_2,g_3]$ and all $p$th powers $g^p$. Define $E(\Omega)$ to be the quotient group $E(\Omega) = F(\Omega)/N(\Omega)$. We shall call $E(\Omega)$ the free group of class two and exponent~$p$ on the set~$\Omega$. Note that if $p=2$ or $\abs{\Omega}=1$ then $E(\Omega)$ is abelian and therefore of class one.

\begin{lemma}
\label{lemma:Eprop}
Let $p$ be an odd prime and $\Omega$ a finite set with at least two elements. The free class two exponent $p$ group $E(\Omega)$ has the following properties:
\begin{enumerate}
\item
\label{enum:Eprop-1}
$E(\Omega)$ is indeed of class two and exponent~$p$.
\item
\label{enum:Eprop-2}
$E(\Omega)$ is special\footnote{In the sense of  \cite[p.~183]{Gorenstein:FiniteGroups}\@.}: its derived and Frattini subgroups coincide with its centre. This is elementary abelian of rank $\binom{\abs{\Omega}}2$, generated by the $[x,y]$ with $x,y \in \Omega$.
\item
\label{enum:Eprop-3}
$\abs{E(\Omega)} = p^n$ for $n = \frac12 \abs{\Omega} (\abs{\Omega}+1)$.
\item
\label{enum:Eprop-4}
$E(a,b)$ is extraspecial of order~$p^3$ and exponent~$p$.
The $\binom{\abs{\Omega}}2$ projection maps $E(\Omega) \twoheadrightarrow E(a,b)$ with $a,b \in \Omega$ detect all elements of $E(\Omega)$.
\end{enumerate}
\end{lemma}

\begin{proof}
\enref{Eprop-1}:
Let $P = \langle x,y \rangle$ be the extraspecial group of order $p^3$ and exponent~$p$. Recall that $P' = \Phi(P) = Z(P)$ is cyclic of order~$p$, generated by $[x,y]$. As $P$ has class two and exponent~$p$, taking two distinct elements $a,b \in \Omega$ induces a group homomorphism $f = f_{a,b} \colon E(\Omega) \rightarrow P$ with $f(a)=x$, $f(b)=y$ and $f(c)=1$ for all $c \in \Omega \setminus \{a,b\}$. Then $f$ is surjective and $f([a,b]) = [x,y] \neq 1$, so $E(\Omega)$ really does have class two and exponent~$p$.

\enref{Eprop-2}: Write $G = E(\Omega)$. Then $\Phi(G) = G'$ since the exponent is~$p$, and $G' \leq Z(G)$ since the class is two. And as all commutators are central we have $[g,hk] =[g,h] \cdot [g,k]$ and $[gh,k] = [g,k][h,k]$. So $G'$ is generated by the $[a,b]$  with $a,b \in \Omega$. If $g \in G \setminus G'$ then $f_{a,b}(g) \in P$ is noncentral for some $a,b \in \Omega$, so $G' = Z(G)$ since $f_{a,b}$ is surjective. Moreover as $f_{a,b}$ kills all commutators except $[a,b] = [b,a]^{-1}$ it follows that the $\binom{\abs{\Omega}}2$ commutators $[a,b]$ are linearly independent in the $\f$-vector space which is the elementary abelian $p$-group $G'$. Part~\enref{Eprop-3} follows since $G/\Phi(G)$ is elementary abelian of rank $\abs{\Omega}$\@. Finally \enref{Eprop-4} now follows by the proof of~\enref{Eprop-2}\@.
\end{proof}

\noindent
From now on we restrict ourselves to the prime three. Set
$H = E(\Omega)$ for $\Omega =\{x_1,x_2,x_3,y_1,y_2,y_3,z_1,z_2,z_3\}$. Define $G$ to be the
semidirect product $G = H \rtimes Q$, where $Q = C_3 \wr C_3 = \langle \xc,\yc,\zc,\tc \rangle$ is the  the Sylow $3$-subgroup of $S_9$, with the following action on~$H$:
\begin{itemize}
\item
The $x$-cycle $\xc$: $x_i^{\xc} = x_{i+1}$, $y_i^{\xc} = y_i$, $z_i^{\xc} = z_i$;
\item
The analogous $y$- and $z$-cycles $\yc$ and $\zc$;
\item
The top-level cycle $x_i^{\tc} =y_i$, $y_i^{\tc} = z_i$, $z_i^{\tc} = x_i$.
\end{itemize}
Since $\abs{H} = 3^{45}$ and $\abs{Q} = 3^4$ we have $\abs{G} = 3^{49}$.

\begin{lemma}
\label{lemma:ZG}
Let $G = H \rtimes Q$ be this $3$-group.
\begin{enumerate}
\item
\label{enum:ZG-1}
$Z(G)$ lies in $H'$ and is elementary abelian of rank two, generated by the $Q$-orbit products of $[x_1,x_2]$ (length $9$) and $[x_1,y_1]$ (length $27$)\@.
\item
\label{enum:ZG-2}
Suppose that $V$ is a faithful $\f[3]G$-module with no central quadratics. Then every quadratic element lies in~$H$.
\end{enumerate}
\end{lemma}

\begin{proof}
Since $H$ is special we have $Z(H) = H'$.
The diagonal action of $Q$ partitions $\Omega^2$ into five orbits: three of length nine with representatives $(x_1,x_1)$, $(x_1,x_2)$ and $(x_2,x_1)$; and two of length $27$ with representatives $(x_1,y_1)$ and $(y_1,x_1)$.
Since $[x,x]=1$ and $[y,x] = [x,y]^{-1}$, this means that $Z(G) \cap H$ is the rank two elementary abelian described in~\enref{ZG-1}\@.

\enref{ZG-2}: Set $K = \langle H, \xc, \yc, \zc \rangle$. If $g \in G \setminus K$ then $g \in K \tc \cup K \tc^2$, so commuting $g$ with $\prod_{i,j} [x_i,y_j][x_i,z_j]$ yields the $Q$-orbit product of $[x_1, y_1]$. This is central, so $g$ and $\prod_{i,j} [x_i,y_j][x_i,z_j]$ are non-quadratic by the Descent Lemma.
And if $k \in K \setminus H$ then without loss of generality $k \in H \langle \yc, \zc \rangle \xc \cup H \langle \yc, \zc \rangle \xc^2$. So its commutator with $\prod_i [x_1,y_i]^{-1}[x_1,z_i]^{-1}[x_2,y_i][x_2,z_i]$ yields $\prod_{i,j} [x_i,y_j][x_i,z_j]$, which is non-quadratic. Hence $k$ is non-quadratic by the Descent Lemma. 
\end{proof}

\begin{proposition}
\label{prop:verify}
Let $G = H \rtimes Q$ be the $3$-group constructed above. Every faithful $\f[3] G$-module $V$ satisfies the Quadratic Conjecture.
\end{proposition}

\begin{proof}
Suppose not. Then $V$ has no central quadratics and yet there is an offender $E \leq G$. By \cite[Theorem 4.5]{ol3} we may assume that $E$ is quadratic and weakly closed in $C_G(E)$ with respect to $G$. So $E \leq H$ by Lemma~\ref{lemma:ZG}\,\enref{ZG-2}\@.

Now observe that $G$ satisfies Hypothesis~\ref{hypo:NA} with $N = H$, $A = H'$. Since $E$ is a weakly closed offender, we will be done by Proposition~\ref{prop:NA} once we can show that $E \in \mathcal{E}$\@. Certainly $E \nleq H'$, by the Normal Abelian Lemma. And observe that by construction of $H$, we have $C_H(g) = \langle g, H' \rangle$ for all $g \in H \setminus H'$. So since $E$ is elementary abelian we have $\abs{E \colon E \cap A} \leq p$. So $E \in \mathcal{E}$, as claimed.
\end{proof}

\section{Constructing the module \texorpdfstring{$V_0$}{V0}}
\label{section:module}
\noindent
As above we restrict our attention to the case $p=3$.
Recall that $E(a,b)$ is extraspecial. We first construct a useful representation of $E(a,b) \times E(c,d)$.

\begin{lemma}
\label{lemma:V1}
Consider the representation $V_1(a,b)$ of $E(a,b)$ in $\GL_3(\f)$ given
by
\begin{xalignat*}{2}
a & \mapsto \begin{pmatrix} 1 & 1 & 0 \\ 0 & 1 & 0 \\ 0 & 0 & 1 \end{pmatrix}
&
b & \mapsto \begin{pmatrix} 1&0&0 \\ 0&1&1 \\ 0&0&1 \end{pmatrix}
\end{xalignat*}
Then $c := [a,b]$ acts as the matrix
$\begin{pmatrix} 1&0&1 \\ 0&1&0 \\ 0&0&1 \end{pmatrix}$.
This representation is faithful, and the elementary abelian
subgroups $\langle a,c \rangle$ and $\langle b,c \rangle$ are quadratic.
\end{lemma}

\begin{proof}
Since $c$ acts nontrivially the representation is faithful.
Every matrix in $\langle a, c \rangle$ is of the form $\begin{pmatrix} 1 & *  & * \\ 0 & 1 & 0 \\ 0 & 0 & 1 \end{pmatrix}$ and is therefore quadratic. Similarly, every matrix in $\langle b, c \rangle$ has the form $\begin{pmatrix} 1 & 0 & * \\ 0 & 1 & * \\ 0 & 0 & 1 \end{pmatrix}$.
\end{proof}

\begin{lemma}
\label{lemma:V2}
Consider the representation $V_2(a_1,a_2;b_1,b_2)$ of
$E(a_1,a_2) \times E(b_1,b_2)$ in $\GL_9(\f)$ given by
\[
V_2(a_1,a_2;b_1,b_2) := V_1(a_1,a_2) \otimes_{\f} V_1(b_1,b_2) \, .
\]
Then $V_2$ is faithful; $\langle a_i, [a_1,a_2] \rangle$ and $\langle b_i, [b_1,b_2] \rangle$ are quadratic for $i=1,2$; and
$[a_1,a_2][b_1,b_2]$ does not act quadratically.
\end{lemma}

\begin{proof}
Let $e_1,e_2,e_3$ be the standard basis of $\f^3$. The centre of $E(a_1,a_2) \times E(b_1,b_2)$ is elementary abelian of rank two generated by $[a_1,a_2]$ and $[b_1,b_2]$. Since $[a_1,a_2]^r [b_1,b_2]^s$ sends $e_3 \otimes e_3$ to $e_3 \otimes e_3 + r e_1 \otimes e_3 + s e_3 \otimes e_1 + rs e_1 \otimes e_1$, the representation is faithful. The subgroups $\langle a_i, [a_1,a_2] \rangle$ and $\langle b_i, [b_1,b_2] \rangle$ are quadratic by Lemma~\ref{lemma:V1}, since each only operates on one factor of $V_2$. For $g = [a_1,a_2][b_1,b_2]$ we have $[e_3 \otimes e_3, g, g] = (e_3 \otimes e_3)(g-1)^2$ and
\[
(e_3 \otimes e_3)(g-1)^2 = (e_1 \otimes e_3 + e_3 \otimes e_1 + e_1 \otimes e_1)(g-1) = 2 e_1 \otimes e_1 \neq 0 \, .
\]
So $g$ is not quadratic.
\end{proof}

\ignore{
The orbits of $(x_1,x_2,y_1,y_2)$,
$(x_1,y_1,x_2,y_2)$ and $(x_1,x_2,y_1,z_1)$ have lengths $27$, $27$ and $81$ respectively. Denote by $\mathcal{B} \subseteq F_4(\Omega)$ the union of these three orbits.

\begin{lemma}
\label{lemma:orbitsCover}
For all $a,b \in \Omega$ with $a \neq b$ there are $c,d \in \Omega$ such that either $(a,b,c,d)$ or $(b,a,c,d)$ lies in the orbit union $\mathcal{B}$.
\end{lemma}
} 

\noindent
We now use several copies of $V_2$ to construct a faithful representation of $G$.
Write $F_4(\Omega)$ for the subset of $\Omega^4$ consisting of those $4$-tuples whose components are pairwise distinct. Note that $Q$ acts on $F_4(\Omega)$ via its diagonal action on $\Omega^4$.
For each $(a,b,c,d) \in F_4(\Omega)$ we may view $V_2(a,b;c,d)$ as a representation of~$H$, with each generator in $\Omega \setminus \{a,b,c,d\}$ acting as the identity.
So the direct sum
\[
V_0 = \bigoplus_{(a,b,c,d) \in F_4(\Omega)} V_2(a,b;c,d) \, ,
\]
is a representation of~$H$, as each summand is; and as $Q$ permutes the summands, $V_0$ is a representation of $G$ too.

\begin{theorem}
\label{thm:V}
$G$ and $V_0$ have the following properties:
\begin{enumerate}
\item
\label{enum:V-1}
$V_0$ is faithful and there are no quadratic elements in the centre.
\item 
\label{enum:V-3}
The subgroup
\[
E := \langle x_1, [x_1,a] \mid a \in \Omega\} \rangle \leq G
\]
is elementary abelian of rank~$9$. It is weakly closed in $C_G(E)$ with
respect to~$G$.
\item
\label{enum:V-4}
$E$ has quadratic action on~$V_0$.
\end{enumerate}
Hence $(G,V_0,E)$ is a counterexample to the Weakly Closed Conjecture.
\end{theorem}

\begin{proof}
\enref{V-1}: If the representation has a kernel then this meets $Z(G)$. Pick $1 \neq g \in Z(G)$. Replacing $g$ by a power if necessary we may assume by Lemma~\ref{lemma:ZG}\,\enref{ZG-1} that $g$ is one of the following: the $Q$-orbit product $c_1$ of $[x_1,x_2]$; the $Q$-orbit product $c_2$ of $[x_1,y_1]$; or $c_1 c_2^r$ for $r = \pm1$.
If $g = c_1$ then by Lemma~\ref{lemma:V2} its action on $V_2(x_1,x_2;y_1,y_2)$ is neither trivial nor quadratic, for the image of $c_1$ in $E(x_1,x_2) \times E(y_1,y_2)$ is $[x_1,x_2][y_1,y_2]$. Similarly the action of $c_2$ on $V_2(x_1,y_1;x_2,y_2)$ is neither trivial nor quadratic. Finally $g = c_1 c_2^r$ has image $[x_1,x_2][y_1,z_1]^r$ in $E(x_1,x_2) \times E(y_1,z_1)$, and so for the action of $g$ on $V_2(x_1,x_2;y_1,z_1)$ we have
\[
[e_3 \otimes e_3, g, g] = (e_3 \otimes e_3)(g-1)^2 = (r e_3 \otimes e_1 + e_1 \otimes e_3 + r e_1 \otimes e_1)(g-1) = 2r e_1 \otimes e_1 \neq 0 \, ,
\]
so the action of $g$ is neither trivial nor quadratic.

\enref{V-3}
From Lemma~\ref{lemma:Eprop} it follows that $E$ is elementary abelian of rank nine. Weakly closed: We have $[E, \langle H, \yc, \zc \rangle] \leq E$, and if
$g \in G \setminus \langle H, \yc, \zc \rangle$ then $x_1^g \in H'a$ for some $a \in \Omega \setminus \{x_1\}$, whence $[x_1,x_1^g] \neq 1$.

\enref{V-4}
$x_1$ quadratic: If no $a_i$ is $x_1$, then $x_1$ acts trivially on $V_2(a_1,a_2;a_3,a_4)$. If $x_1$ is amongst the~$a_i$, then $x_1$ acts quadratically by Lemma~\ref{lemma:V2}. In particular the action on $V_2(x_1,x_2;y_1,y_2)$ is nontrivial.

$E$ quadratic:
We have $E = \langle x_1 \rangle \times F$ for $F = E \cap H'$.
Since $[x_1,g][x_1,h] = [x_1,gh]$ for $g,h \in H$ it follows that $F = \{ [x_1,g] \mid g \in H \}$. So for each $1 \neq e \in E$ there is an $a \in \Omega \setminus {x_1}$ with $e \in F_a = \langle x_1, [x_1,a] \rangle$\@. Since $[x_1,a] \neq 1$ and $x_1$ is quadratic, it follows from the Descent Lemma that $F_a$ is quadratic too. So $e$ is quadratic.
\end{proof}


\end{document}